\documentclass[oneside,american,english,british]{amsart}
\usepackage[T1]{fontenc}
\usepackage{geometry}
\usepackage{amsthm}
\usepackage{amstext}
\usepackage{amssymb}
\usepackage{setspace}
\usepackage{units}
\usepackage{babel}
\usepackage[latin9,cp1255]{inputenc}
\PassOptionsToPackage{normalem}{ulem}
\usepackage{ulem}

\makeatletter
\numberwithin{equation}{section}
\numberwithin{figure}{section}

\theoremstyle{plain}
\newtheorem{thm}{\protect\theoremname}
  \theoremstyle{definition}
  \newtheorem{defn}[thm]{\protect\definitionname}
  \theoremstyle{definition}
  \newtheorem{problem}[thm]{\protect\problemname}
  \theoremstyle{plain}
  \newtheorem{prop}[thm]{\protect\propositionname}
  \theoremstyle{plain}
  \newtheorem{lem}[thm]{\protect\lemmaname}
  \theoremstyle{plain}
  \newtheorem{cor}[thm]{\protect\corollaryname}
  \theoremstyle{remark}
  \newtheorem{rem}[thm]{\protect\remarkname}
  \theoremstyle{definition}
  \newtheorem{example}[thm]{\protect\examplename}
  \theoremstyle{remark}
  \newtheorem{claim}[thm]{\protect\claimname}
  \theoremstyle{plain}
  \newtheorem{fact}[thm]{\protect\factname}
\numberwithin{thm}{section}

\makeatother
\usepackage{babel}
  \addto\captionsamerican{\renewcommand{\claimname}{Claim}}
  \addto\captionsamerican{\renewcommand{\corollaryname}{Corollary}}
  \addto\captionsamerican{\renewcommand{\definitionname}{Definition}}
  \addto\captionsamerican{\renewcommand{\examplename}{Example}}
  \addto\captionsamerican{\renewcommand{\factname}{Fact}}
  \addto\captionsamerican{\renewcommand{\lemmaname}{Lemma}}
  \addto\captionsamerican{\renewcommand{\problemname}{Problem}}
  \addto\captionsamerican{\renewcommand{\propositionname}{Proposition}}
  \addto\captionsamerican{\renewcommand{\remarkname}{Remark}}
  \addto\captionsamerican{\renewcommand{\theoremname}{Theorem}}
  \addto\captionsbritish{\renewcommand{\claimname}{Claim}}
  \addto\captionsbritish{\renewcommand{\corollaryname}{Corollary}}
  \addto\captionsbritish{\renewcommand{\definitionname}{Definition}}
  \addto\captionsbritish{\renewcommand{\examplename}{Example}}
  \addto\captionsbritish{\renewcommand{\factname}{Fact}}
  \addto\captionsbritish{\renewcommand{\lemmaname}{Lemma}}
  \addto\captionsbritish{\renewcommand{\problemname}{Problem}}
  \addto\captionsbritish{\renewcommand{\propositionname}{Proposition}}
  \addto\captionsbritish{\renewcommand{\remarkname}{Remark}}
  \addto\captionsbritish{\renewcommand{\theoremname}{Theorem}}
  \addto\captionsenglish{\renewcommand{\claimname}{Claim}}
  \addto\captionsenglish{\renewcommand{\corollaryname}{Corollary}}
  \addto\captionsenglish{\renewcommand{\definitionname}{Definition}}
  \addto\captionsenglish{\renewcommand{\examplename}{Example}}
  \addto\captionsenglish{\renewcommand{\factname}{Fact}}
  \addto\captionsenglish{\renewcommand{\lemmaname}{Lemma}}
  \addto\captionsenglish{\renewcommand{\problemname}{Problem}}
  \addto\captionsenglish{\renewcommand{\propositionname}{Proposition}}
  \addto\captionsenglish{\renewcommand{\remarkname}{Remark}}
  \addto\captionsenglish{\renewcommand{\theoremname}{Theorem}}
  \providecommand{\claimname}{Claim}
  \providecommand{\corollaryname}{Corollary}
  \providecommand{\definitionname}{Definition}
  \providecommand{\examplename}{Example}
  \providecommand{\factname}{Fact}
  \providecommand{\lemmaname}{Lemma}
  \providecommand{\problemname}{Problem}
  \providecommand{\propositionname}{Proposition}
  \providecommand{\remarkname}{Remark}
  \providecommand{\theoremname}{Theorem}

\begin{document}
\begin{onehalfspace}

\title{Borel Canonization of Analytic Sets with Borel Sections}

\author{Ohad Drucker}

\begin{abstract}
Given an analytic equivalence relation, we tend to wonder whether
it is Borel. When it is non Borel, there is always the hope it will
be Borel on a ``large'' set -- nonmeager or of positive measure.
That has led Kanovei, Sabok and Zapletal to ask whether every proper
$\sigma$ ideal satisfies the following property: given $E$ an analytic
equivalence relation with Borel classes, there exists a set $B$ which
is Borel and $I$-positive such that $E\restriction_{B}$ is Borel. We propose a related problem -- does every proper $\sigma$ ideal
satisfy: given $A$ an analytic subset of the plane with Borel sections,
there exists a set $B$ which is Borel and $I$-positive such that
$A\cap(B\times\omega^{\omega})$ is Borel. We answer positively when
a measurable cardinal exists, and negatively in $L$, where no proper
$\sigma$ ideal has that property. Assuming $\omega_{1}$ is inaccessible
to the reals but not Mahlo in $L$, we construct a ccc $\sigma$ ideal
$I$ not having this property -- in fact, forcing with $I$ adds a
non Borel section to a certain analytic set with Borel sections, and
a non Borel class to a certain analytic equivalence relation with
Borel classes. Various counterexamples are given for the case of a $\mathbf{\Delta_{2}^{1}}$
equivalence relation as well as for the case of an improper ideal.
\end{abstract}

\maketitle
\section{Introduction}

\subsection{Borel Canonization of Analytic Equivalence Relations}

Analytic equivalence relations are common in the world of mathematics,
and given such an equivalence relation, one of the first questions
traditionally asked is -- ``is it Borel?''. A negative answer used
to convince us that the equivalence relation is relatively complicated,
but a new point of view proposed by Kanovei, Sabok and Zapletal has
opened the way to a somewhat more optimistic conclusion. We all know
that Lebesgue measurable functions are ``almost continuous'', analytic
sets are Borel modulo meager sets and colorings of natural numbers
are ``almost'' trivial. We can then hope that even the non Borel
analytic equivalence relations are Borel on a substantial set -- which
leads to the following question:

\selectlanguage{american}%
\begin{problem}
\label{general Borel canonization e.r. problem for measure/meager}Given
an analytic equivalence relation $E$ on a Polish space $X$, does
there exist a positive measure (or non-meager, or uncountable) Borel
set $B$ such that $E$ restricted to $B$ is Borel ?
\end{problem}

We can use the notion of a $\sigma$-ideal to state a more general
problem. Given a $\sigma$-ideal $I$ , we will say that $A$ is
an $I$-positive set if $A\notin I$, an $I$-small set if $A\in I$
, and a co-$I$ set if $X-A\in I$. The above mentioned problem
involved the existence of an $I$-positive set for the null ideal,
the meager ideal and the countable ideal. We restate it for all $\sigma$
-ideals:

\begin{problem}
Given an analytic equivalence relation $E$ on a Polish space $X$
and a $\sigma$-ideal $I$ , does there exist an $I$-positive
Borel set $B$ such that $E$ restricted to $B$ is Borel ?
\end{problem}

Unfortunately, that problem has a negative answer, and further assumptions
had to be made -- both on the equivalence relation $E$ and on the
$\sigma$-ideal $I$ (see section \ref{sec:Examples-and-Counterexamples}).
We recall that for a $\sigma$-ideal $I$, $\mathbb{P}_{I}$ is
the partial order of Borel $I$-positive subsets, ordered by inclusion.
We say that $I$ is \textsl{proper} if the associated forcing notion
$\mathbb{P}_{I}$ is proper. Then Kanovei, Sabok and Zapletal have
asked the following:

\begin{problem}
\cite{ksz} \label{Borel canonization e.r.}\textsl{Borel canonization
of analytic equivalence relations with Borel classes:} Given an analytic
equivalence relation $E$ on a Polish space $X$, all of its classes
Borel, and a proper $\sigma$-ideal $I$, does there exist an $I$-positive Borel set $B$ such that $E$ restricted to $B$ is Borel?

They have shown the answer to be positive for two important classes
of analytic equivalence relations with Borel classes: orbit equivalence
relation, and countable equivalence relations (proofs are given
in the next section). The problem in its full generality remained
open.
\end{problem}

\subsection{Borel Canonization of Analytic Sets with Borel Sections}

Let $E$ be an analytic equivalence relation on $X$ Polish, all of
its classes Borel. The equivalence relation $E$ is a subset of $X^{2}$ with Borel sections.
Given $B\subseteq X$ Borel, $E\restriction B$ is Borel is equivalent,
by the very definition, to $E\cap(B\times B)$ being Borel. That simple
observation leads to the following variants of Borel canonization:

\begin{defn}
Let $X$ be Polish, and $I$ a $\sigma$-ideal on $X$. 
\begin{enumerate}
\item We say that $I$ has \textsl{square Borel canonization of analytic
sets with Borel sections }if for any\textsl{ }$A\subseteq X^{2}$
an analytic set with vertical Borel sections, there exists an $I$-positive Borel set $B$ such that $A\cap(B\times B)$ is Borel.
\item We say that $I$ has \textsl{rectangular Borel canonization of analytic
sets with Borel sections }if for any\textsl{ }$A\subseteq X^{2}$
an analytic set with vertical Borel sections, there exists an $I$-positive Borel set $B$ such that $A\cap(B\times X)$ is Borel.
\item We say that $I$ has\textbf{\textsl{ }}\textsl{strong}\textbf{\textsl{
}}\textsl{square Borel canonization of analytic sets with Borel sections
}if for any\textsl{ }$A\subseteq X^{2}$ an analytic set with vertical
Borel sections, there exists a co-$I$ Borel set $B$ such that
$A\cap(B\times B)$ is Borel.
\item We say that $I$ has\textbf{\textsl{ }}\textsl{strong}\textbf{\textsl{
}}\textsl{rectangular Borel canonization of analytic sets with Borel
sections }if for any\textsl{ }$A\subseteq X^{2}$ an analytic set
with vertical Borel sections, there exists a co-$I$ Borel set $B$
such that $A\cap(B\times X)$ is Borel.
\end{enumerate}
\end{defn}

In what follows, we will simply say: ``$I$ has square Borel canonization''
, etc. Rectangular Borel canonization implies square Borel
canonization, which implies Borel canonization of analytic equivalence
relations with Borel classes. We do not know whether any of the inverse implications
are true.

\begin{rem}
\label{remark on ccc and borel canonization}For ccc ideals, the strong
Borel canonization and the weak Borel canonization are equivalent.
The strong Borel canonization of general proper ideals is false -- 
see \cite{ikegami-2nd} proposition 17.
\end{rem}

When considering the square and rectangular Borel canonizations, there
is no difference between analytic and coanalytic sets:

\begin{claim}
$I$ has square Borel canonization of analytic sets with Borel sections
if and only if $I$ has square Borel canonization of coanalytic sets
with Borel sections (and the same for rectangular, strong square and
strong rectangular Borel canonizations).
\end{claim}
\begin{proof}
Consider the complement.
\end{proof}

Hence in that context, analytic sets and coanalytic sets are basically
the same object.

Albeit being a new notion, strong rectangular Borel canonization has been studied in the past by Fujita in
\cite{fujita} and by Ikegami in \cite{ikegami-1st} and \cite{ikegami-2nd} ,
culminating in the following result:

\begin{thm}
\label{thm:(-Ikegami-)} (Ikegami \cite{ikegami-2nd}) Let $I$ be a Borel
generated $\sigma$-ideal such that $\mathbb{P}_{I}$ is strongly
arboreal, provably ccc and $\mathbf{\Sigma_{1}^{1}}.$ Then the following
are equivalent:
\begin{enumerate}
\item $I$ has strong rectangular Borel canonization.
\item $\mathbf{\Sigma_{2}^{1}}$ sets are measurable with respect to $I_{\text{ }}$,
which is: For $A$ $\mathbf{\Sigma_{2}^{1}}$ there is $B$ Borel
such that $A\triangle B\in I$.
\end{enumerate}
\end{thm}

We say that $I$ is \textit{Borel generated} if any $A\in I$ is contained
in an $I$-small Borel set. We say that $I$ is \textit{provably
ccc} if $ZFC$ proves that $I$ is ccc. The notions of ``strongly
arboreal'' and ``$\mathbf{\Sigma_{1}^{1}}$ forcing'' will be defined
in the following section. For now, we will only say these are assumptions
on the presentability and definability of $\mathbb{P}_{I}$, satisfied
by, for example, the meager ideal and the null ideal. Hence, one learns
from the theorem that the meager ideal has strong rectangular Borel
canonization if and only if $\mathbf{\Sigma_{2}^{1}}$ sets have the
Baire property, and the null ideal has strong rectangular Borel canonization
if and only if $\mathbf{\Sigma_{2}^{1}}$ sets are Lebesgue measurable.

This paper will focus on general $\sigma$-ideals with minimal assumptions
on definability and presentability. It is therefore interesting and
illuminating to compare our results with Ikegami's results.

\subsection{The results of this paper}

The problem of Kanovei, Sabok and Zapletal can be restated as:

\begin{problem}
Do all proper ideals $I$ have Borel canonization of analytic equivalence
relation with Borel classes?
\end{problem}

We focus our paper at the following related problem:

\begin{problem}
\label{main_problem}Do all proper ideals $I$ have rectangular Borel
canonization of analytic sets with Borel sections?
\end{problem}

Section 2 reviews definitions and facts which we use in this paper,
and elaborates on previous results about Borel canonization.

In section 3 we define a notion of $\omega_{1}$-rank for analytic
sets with Borel sections. We use the rank to prove:

\begin{thm}
Assume a measurable cardinal exists. Then proper ideals have rectangular Borel
canonization and ccc ideals have strong rectangular Borel canonization.
\end{thm}

We say that \textit{$\omega_{1}$ is inaccessible to the reals} if
for every $z$ real, $\omega_{1}^{L[z]}<\omega_{1}$.

\begin{thm}
Assume $\omega_{1}$ is inaccessible to the reals, and $I$ is ccc
in $L[z]$ for any real $z$. Then $I$ has strong rectangular Borel canonization
of analytic sets all of whose sections are $\mathbf{\Pi_{\gamma}^{0}}$ for some
$\gamma<\omega_{1}$.
\end{thm}

Section 4 presents examples and counterexamples, mainly demonstrating
the necessity of assuming the properness of $I$ in problem \ref{Borel canonization e.r.}.
For example:

\begin{prop}
\label{I generated by thin e.r.}Let $E$ be analytic with uncountably
many classes but not perfectly many. Let $I_{E}$ be the $\sigma$-ideal generated by the equivalence classes. Then for any $B$ Borel
$I_{E}$-positive, $E\restriction_{B}$ is non Borel.
\end{prop}

In \cite{my-psp}, we show that $I_{E}$ as above is never proper,
hence that proposition does not provide a negative answer to the problem
of Borel canonization (problem \ref{Borel canonization e.r.}).

Extending our discussion to $\mathbf{\Delta_{2}^{1}}$ sets, we remark
that $L$ demonstrates a strong form of failure of Borel canonization:

\begin{thm}
\cite{chan}
In $L$, $\sigma$-ideals do not have Borel canonization of $\mathbf{\Delta_{2}^{1}}$
equivalence relations with Borel classes.
\end{thm}

The last section presents counterexamples to rectangular Borel canonization,
both in $L$ and in much larger universes:

\begin{prop}
In $L$, proper ideals do not have rectangular Borel canonization
of analytic sets with Borel sections. The same is true for $L[z]$
where $z$ is a real. 
\end{prop}

\begin{thm}
If $\omega_{1}$ is inaccessible to the reals and is not Mahlo in
$L$, then there is a ccc ideal not having rectangular Borel canonization
of analytic sets with Borel sections. Moreover, $\mathbb{P}_{I}\Vdash A_{x_{G}}\ non\ Borel$
for some $A$ analytic with Borel sections.
\end{thm}

\begin{cor}
Rectangular Borel canonization for ccc ideals implies that $\omega_1$ is inaccessible to the reals and Mahlo in L.
\end{cor}

Non absoluteness of ``all sections / classes are Borel'' is further
demonstrated by the following proposition:

\begin{prop}
There is an analytic equivalence relation $E$ such that:
\begin{enumerate}
\item If $\omega_{1}$ is inaccessible to the reals and is not Mahlo in
$L$, then all $E$ classes are Borel and there is a ccc ideal $I$
such that
\[
\mathbb{P}_{I}\Vdash[x_{G}]\ is\ non\ Borel.
\]
\item If $\omega_{1}$ is inaccessible to the reals, then all $E$ classes
are Borel, while in $L$ there is a non Borel class.
\end{enumerate}
\end{prop}

The problem of square Borel canonization is sometimes discussed in
this paper but the consistency of a negative answer remains open.
The same applies for the problem of Borel canonization of equivalence
relations.

Chan, in \cite{chan}, has independently obtained much of the above
results using similar techniques. He has been working with equivalence
relations, but his proofs perfectly fit in the context of rectangular
Borel canonization. In particular, he has shown that all proper ideals have rectangular
Borel canonization if there exist sharps
for all reals and for a few more sets associated with the forcing notions
of proper ideals.

\subsection{Acknowledgments}

This research was carried out under the supervision of Menachem Magidor,
and would not be possible without his elegant ideas and deep insights.
The author would like to thank him for his dedicated help. The author
would also like to thank Marcin Sabok for introducing him with the problem of Borel canonization and for hours of helpful discussions about the subject,
and to thank William Chan for sharing and discussing his results and thoughts,
and for reading the first draft of this paper.

\section{Preliminaries}

\subsection{Forcing with Ideals}

The basics of forcing can be found in \cite{jech}. We remind that
a forcing notion satisfies the \textit{countable chain condition},
or is \textit{ccc}, when every antichain is countable. A wider class
of forcing notions is the proper ones:

\begin{defn}
Let $M$ be an elementary submodel of $H_{\theta},$ and $\mathbb{P}$
a forcing notion. A condition $q\in\mathbb{P}$ is a \textit{master
condition }over $M$ if for every $D\subseteq\mathbb{P}$ dense such
that $D\in M$, $q\Vdash\dot{G}\cap\check{D}\cap\check{M}\neq\emptyset.$
A forcing notion $\mathbb{P}$ is \textit{proper} if for every $\theta$
large enough, every countable elementary submodel $M$ of $H_{\theta}$
such that $\mathbb{P}\in M$, and every $p\in\mathbb{P}\cap M$, there
is a $q\leq p$ which is a master condition over $M$.
\end{defn}
\selectlanguage{english}%

Proper forcing notions preserve $\omega_{1}$, but there are $\omega_{1}$-preserving forcing notions which are not proper.

The subject of forcing with ideals was thoroughly investigated by
Zapletal in \foreignlanguage{american}{\cite{forcing-idealized}. We review here
some of the most important notions and facts.}

\selectlanguage{american}%
Recall that for $I$ a $\sigma$-ideal on a Polish space $X$, $\mathbb{P}_{I}$
is the partial order of Borel $I$-positive sets ordered by inclusion.

\begin{prop}
The poset $\mathbb{P}_{I}$ adds an element $\dot{x_{G}}$ of the
Polish space $X$ such that for every Borel set $B\subseteq X$ coded
in the ground model, $B\in G\Leftrightarrow\dot{x_{G}}\in B$.
\end{prop}

Given $M$ a transitive model of $ZFC$, we say that $x\in X$ is
generic over $M$ if 
\[
\{B\in\mathbb{P}_{I}\cap M\ :\ x\in B\}
\]
 is a generic filter over $M$.

A $\mathbb{P}_{I}$-generic point avoids all Borel $I$-small
sets of the ground model. When $\mathbb{P}_{I}$ is ccc, this is a
complete characterization of the generic points:

\begin{prop}
Let $M$ be a transitive model of $ZFC$, and $I\in M$ a $\sigma$-ideal on $X$ such that $M\models\mathbb{P}_{I}$ is ccc. Then $x\in X$
is $\mathbb{P}_{I}$ generic over $M$ if and only if for every $I$-small set $B$ coded in $M$, $x\notin B$.
\end{prop}
\selectlanguage{english}%

Common in this paper is forcing over a model which is well founded
but not transitive -- what we really mean by that is forcing over its
transitive collapse.

\begin{cor}
Let $I$ be a $\sigma$-ideal such that $\mathbb{P}_{I}$ is ccc,
and $M$ a countable elementary submodel of a large enough $H_{\theta}$
such that $\mathbb{P}_{I}\in M$ and $B\in\mathbb{P}_{I}\cap M$ .
The set of elements of $B$ which are generic over $M$ is co-$I$
in $B$.
\end{cor}

We say that $I$ is ccc if $\mathbb{P}_{I}$ is ccc, and that $I$
is proper if $\mathbb{P}_{I}$ is proper. Properness of $\mathbb{P}_{I}$
can be phrased in terms of the set of $M$-generics:

\begin{prop}
$\mathbb{P}_{I}$ is proper if and only if for every $M$ a countable
elementary submodel of a large enough $H_{\theta}$ such that $\mathbb{P}_{I}\in M$
and for every $B\in\mathbb{P}_{I}\cap M$, the set of elements of
$B$ which are generic over $M$ is $I$-positive. 
\end{prop}
\selectlanguage{english}%

As an example, we use the above characterization to show:

\begin{cor}
Forcing with a proper ideal preserves $\omega_{1}$.
\end{cor}

\begin{proof}
Assume otherwise, and fix $B\in\mathbb{P}_{I}$ such that
\[
B\Vdash\dot{f}:\omega\to\omega_{1}\ onto.
\]
Let $M$ be a countable elementary submodel of a large enough $H_{\theta}$
such that $\mathbb{P}_{I}\in M$ and $B\in M$. Let $C\subseteq B$
be the $I$-positive Borel set of the $M$-generics of $B$. Now
force an $x_{G}\in C$ which is $\mathbb{V}$-generic. In $\mathbb{V}[x_{G}]$,
$x_{G}$ is still an element of $C$, and $C$ is defined as the set
of $M$-generics, so $x_{G}$ is $M$-generic. At the same time,
$\mathbb{V}[x_{G}]$ interprets $\dot{f}$ as a map from $\omega$
onto $\omega_{1}^{\mathbb{V}}$. The interpretation of $\dot{f}$
in $M[x_{G}]$ should be consistent with the one of $\mathbb{V}[x_{G}]$.
Since both models have the same interpretation of $\omega,$ $\dot{f}^{M[x_{G}]}=\dot{f}^{\mathbb{V}[x_{G}]}$,
and in particular $M[x_{G}]$ contains $\omega_{1}^{\mathbb{V}}$.
But $M[x_{G}]$ and $M$ have the same ordinals, and $M$ is countable
 -- a contradiction.
\end{proof}

This paper is about Borel canonization, and we make an informal claim
that Borel canonization and genericity over $M$ are strongly connected.
Intuitively, the generic elements over $M$ are well described by
the countably many conditions in $\mathbb{P}_{I}\cap M$, so one can
hope that restricting equivalence relations to the set of $M$-generics
will make the equivalence relation more definable. The above propositions
assure that when $I$ is proper, the set of generics is indeed big:
$I$-positive in general and co-$I$ for ccc ideals. As a result,
when $I$ is proper, for Borel canonization it will be enough to show
that the equivalence relation is simpler on the set of generics. For
improper ideals, a completely different approach should probably be
taken.

\subsection{Borel canonization of orbit equivalence relations and countable equivalence
relations}

We now give the proofs of the two Borel canonization results of Kanovei,
Sabok and Zapletal \cite{ksz}. The first one is rewritten using
the notion of Hjorth rank (see \cite{hjorth-paper,my-hjorth}), and the second
one is generalized so that it shows rectangular Borel canonization
of analytic sets with countable sections:

\begin{thm}
Proper ideals have Borel canonization of orbit equivalence
relations.
\end{thm}

\begin{proof}
Let $G$ be a Polish group acting on a Polish space $X$, and $I$
a proper $\sigma$-ideal. We find $C$ Borel and $I$-positive
such that $\left(E_{G}^{X}\right)\restriction_{C}$ is Borel. Let
$\delta$ be the Hjorth rank associated with the action of $G$ on
$X$ . Fix $\theta$ large enough and $M\preceq H_{\theta}$ an elementary
submodel containing all the relevant information. Let $C$ be the
$I$-positive Borel set of $M$-generics, and  let $x\in C$ be $M$-generic. Then
\[
M[x]\models\delta(x)\leq\alpha
\]
for some $\alpha<\omega_{1}^{M[x]}=\omega_{1}^{M}$ . The rank $\delta$
has a Borel definition, hence $\mathbb{V}\models\delta(x)\leq\alpha$
as well. We have thus proved that the Hjorth rank on $C$ is uniformly
bounded below $\omega_{1}^{M}$, hence $\left(E_{G}^{X}\right)\restriction_{C}$
is Borel.
\end{proof}

\begin{thm}
Proper ideals have rectangular Borel canonization of analytic sets
with countable sections.
\end{thm}

\begin{proof}
Fix $I$ proper and $A$ an analytic subset of the plane with countable
sections. Recall that a $\Sigma_{1}^{1}(x)$ set is countable if and
only if all its elements are hyperarithmetic in $x$. One can then
show that ``all sections are countable'' is still true in generic
extensions. Use 2.3.1 of \cite{forcing-idealized} to find $B\in\mathbb{P}_{I}$
and a Borel $f:B\to X^{\omega}$ such that $B\Vdash f(x_{G})\ enumerates\ A_{x_{G}}.$

Fix $\theta$ large enough and $M\preceq H_{\theta}$ an elementary
submodel containing all the relevant information (including $f$ and
$B$). Let $C\subseteq B$ be the $I$-positive Borel set of $M$-generics, and let $x\in C$ be $M$-generic. Then
\[
M[x]\models f(x)\ enumerates\ A_{x},
\]
which is, 
\[
M[x]\models\forall y\ (y\in A_{x})\Rightarrow\exists n\in\omega\ (f(x))(n)=y.
\]
 That statement is $\Pi_{1}^{1},$ so it must be true in $\mathbb{V}$
as well -- which is, $(A_{x})^{\mathbb{V}}\subseteq M[x]$ . On the
other hand, if 
\[
(f(x))(n)=y
\]
 then $y\in M[x]$ and $M[x]\models y\in A_{x},$ hence $\mathbb{V}$
thinks the same.

The above results in a Borel definition of $A\cap(C\times X)$:
For $x\in C$ and $y\in X$,
\[
(x,y)\in A\Leftrightarrow\exists n\in\omega\ (f(x))(n)=y.\qedhere
\]
\end{proof}

\subsection{$\mathbb{P}_{I}$-measurable sets}

We quickly review the definitions and results of \cite{ikegami-1st}. 

A\textit{ tree on $\omega$} is a subset of $\omega^{<\omega}$ closed
under initial segments. The set of branches through $T$ is
\[
[T]=\{f\in\omega^{\omega}\ :\ \forall n\ f\restriction_{n}\in T\}.
\]
For a forcing notion $\mathbb{P}$ , we say that $\mathbb{P}$ is
\textit{strongly arboreal} if the conditions of $\mathbb{P}$ are
perfect trees on $\omega$, and 
\[
T\in\mathbb{P},\ s\in T\Rightarrow T\restriction_{s}\in\mathbb{P}
\]
where $T\restriction_{s}=\{t\ :\ t\in T;\ t\supseteq s\ or\ t\subseteq s\}$.
A strongly arboreal forcing notion adds a generic real $x_{G}$ such
that $\mathbb{V}[x_{G}]=\mathbb{V}[G]$. The generic $x_{G}$ is exactly
that real in $\mathbb{V}[G]$ that is a branch through all trees of
the generic filter $G$. We will abuse notation and say that $\mathbb{P}$
is strongly arboreal if it has a presentation that is strongly arboreal.

Let $\mathbb{P}$ be strongly arboreal. A set of reals $A$ is \textit{$\mathbb{P}$-null} if every $T\in\mathbb{P}$ can be extended to $T'$ such that
$[T']\cap A=\emptyset.$ We denote by $N_{\mathbb{P}}$ the collection
of $\mathbb{P}$-null sets, and by $I_{\mathbb{P}}$ the $\sigma$-ideal generated by the $\mathbb{P}$-null sets.

We say that a set of reals $A$ is \textit{$\mathbb{P}$-measurable}
if every $T\in\mathbb{P}$ can be extended to $T'$ such that $[T']\cap A\in I_{\mathbb{P}}$
or $[T']-A\in I_{\mathbb{P}}$.

Given a ccc ideal $I$ such that $\mathbb{P}_{I}$ is strongly arboreal,
one can consider all the above notions for $\mathbb{P}_{I}$. The
resulting $\sigma$-ideal $I_{\mathbb{P}_{I}}$ will be the one
generated by the Borel $I$-small sets. Hence when $I$ is Borel
generated, $I=I_{\mathbb{P}_{I}}.$ In that case, \emph{$\mathbb{P}_I$-measurable} is what is usually called \emph{$I$-measurable} (see \cite{brendle-lowe}). 

A $\sigma$-ideal $I$ is said to be $\mathbf{\Sigma_{n}^{1}}$ or $\mathbf{\Pi^1_n}$
if the set of Borel codes of $I$-small sets is. The term ``\textit{provably ccc}'' refers to $\sigma$-ideals which
are ccc in all models of $ZFC.$

\begin{thm}
\label{delta_1_2_and_L_generics}Let $I$ be a provably
ccc, provably $\mathbf{\Delta_{2}^{1}}$ and Borel generated $\sigma$-ideal such that $\mathbb{P}_{I}$ is strongly arboreal.  The following are equivalent:
\begin{enumerate}
\item Every $\mathbf{\Delta_{2}^{1}}$ set of reals is $\mathbb{P}_{I}$-measurable.
\item For any real $a$ and $T\in\mathbb{P},$ there is an $L[a]$ generic in $[T]$. 
\end{enumerate}
\end{thm}

\begin{thm}
\label{sigma_1_2_and_L_generics}Let $I$ be a provably
ccc, provably $\mathbf{\Delta_{2}^{1}}$ and Borel generated $\sigma$-ideal such that $\mathbb{P}_{I}$ is strongly arboreal. The following are equivalent:
\begin{enumerate}
\item Every $\mathbf{\Sigma_{2}^{1}}$ set of reals is $\mathbb{P}_{I}$-measurable.
\item For any real $z$, the set of $\mathbb{P}_I$ generics over $L[z]$ is co-$I$.
\end{enumerate}
\end{thm}

The above were used in \cite{ikegami-2nd} to obtain:

\begin{thm}
\label{thm_Ikegami_more_precise}Let $I$ be a provably
ccc, $\mathbf{\Sigma_{1}^{1}}$ and Borel generated $\sigma$-ideal such that $\mathbb{P}_{I}$ is strongly arboreal. Then the following are equivalent:
\begin{enumerate}
\item $I$ has strong rectangular Borel canonization.
\item Every $\mathbf{\Sigma_{2}^{1}}$ set of reals is $\mathbb{P}_{I}$-measurable.
\end{enumerate}
\end{thm}

In fact, Ikegami has shown the following:

\begin{lem}
\label{main_lemma_ikegami}Let $A$ be a $\Sigma_{1}^{1}(z)$ subset
of the plane with Borel sections. Then for $B$ a Borel subset of
$L[z]$-generics, $A\cap(B\times\omega^{\omega})$ is Borel.
\end{lem}

Lemma \ref{main_lemma_ikegami} together with theorem \ref{sigma_1_2_and_L_generics} proves
$(2)\Rightarrow(1)$ of theorem \ref{thm_Ikegami_more_precise}.

\section{Ranks for Analytic Sets with Borel Sections}

Let $A$ be an analytic subset of $(\omega^{\omega})^{2}$. There
exists a tree $T\subseteq\omega^{\omega}\times\omega^{\omega}\times\omega^{\omega}$
such that
\[
(x,y)\in A\Leftrightarrow T_{xy}\notin WF.
\]

For $\alpha<\omega_{1}$, define:
\[
(x,y)\in A_{\alpha}\Leftrightarrow T_{xy}\notin WF_{\alpha}.
\]

The sequence $A_{\alpha}$ is decreasing, $A_{\delta}=\cap_{\alpha<\delta}A_{\alpha}$
for $\delta$ limit, and

\[
A=\cap_{\alpha<\omega_{1}}A_{\alpha}.
\]

\begin{defn}
For $x\in\omega^{\omega}$, the\textit{ rank of $x$} , $\delta(x)$,
is the least $\alpha$ such that $A_{x}=(A_{\alpha})_{x}$, if such
an $\alpha$ exists, and $\infty$ if there is no such $\alpha$.
\end{defn}

\begin{prop}
\label{boundedness_for_sections}If $A_{x}$ is Borel, then there
is $\alpha<\omega_{1}$ such that $A_{x}=(A_{\alpha})_{x}$.
\end{prop}

\begin{proof}
Since 
\[
(^{\sim}A)_{x}=\{y\ :\ (x,y)\notin A\}=\{y\ :\ T_{xy}\in WF\}
\]
 is a Borel set, its image under $y\to T_{xy}$ is an analytic subset
of $WF$. By the boundedness theorem for $WF$, its image is contained
in $WF_{\alpha}$ for some countable $\alpha$, which is:
\[
y\in A_{x}\Leftrightarrow T_{xy}\notin WF_{\alpha}\Leftrightarrow y\in(A_{\alpha})_{x}
\]
as we wanted to show.
\end{proof}

\begin{prop}
\label{definability_of_rank}The set $\Delta=\{(x,f)\ :\ f\in WO,\ \delta(x)\leq ot(f)\}$
is $\Pi_{2}^{1}$. The set $\{x\ :\ A_{x}\ Borel\}$ is $\Sigma_{3}^{1}$.
\end{prop}

\begin{proof}
$f\in WO$ is $\Pi_{1}^{1}$. The rank of $x$ is less than the order
type of $f$ if and only if 
\[
\forall z:\ T_{xz}\in WF\Leftrightarrow T_{xz}\in WF_{ot(f)}
\]
which is $\Pi_{2}^{1}$. For $x\in X$, $A_{x}$ is Borel if and only
if $\exists f$ such that $(x,f)\in\Delta$, which is $\Sigma_{3}^{1}$.
\end{proof}

\begin{prop}
Let $B\subseteq X$ be a Borel set. Then $A\cap(B\times X)$ is Borel
if and only if there is an $\alpha<\omega_{1}$ such that for all
$x\in B$, $\delta(x)<\alpha$.
\end{prop}

The proof uses the boundedness theorem for $WF$ in the same way used
in the proof of proposition \ref{boundedness_for_sections}.

When one considers square Borel canonization, or Borel canonization
of equivalence relations, the rank $\delta$ has to be relativized:

\begin{defn}
For $x\in\omega^{\omega}$ and $B$ Borel, \textit{the rank of $x$
with respect to $B$} , $\delta_{B}(x)$, is the least $\alpha$ such
that $A_{x}\cap B=(A_{\alpha})_{x}\cap B$, if such an $\alpha$ exists,
and $\infty$ if there is no such $\alpha$.
\end{defn}

\begin{prop}
Let $B\subseteq X$ be a Borel set. The set $\Delta=\{(x,f)\ :\ f\in WO,\ \delta_{B}(x)\leq ot(f)\}$
is $\Pi_{2}^{1}$. The set $\{x\ :\ A_{x}\cap B\ is\ Borel\}$ is
$\Sigma_{3}^{1}$.
\end{prop}

\begin{prop}
Let $B\subseteq X$ be a Borel set. Then $A\cap(B\times B)$ is Borel
if and only if there is an $\alpha<\omega_{1}$ such that for all
$x\in B$, $\delta_{B}(x)<\alpha$.
\end{prop}

We remark that the rank is not canonical and depends on the choice
of the tree $T$. However, all we will need for our results is the
mere existence of such a rank.

\begin{example}
Analytic equivalence relations with Borel classes are convenient examples
of analytic subsets of the plane with Borel sections. Let us find
a rank for two such equivalence relations:
\begin{enumerate}
\item Given $x,y\in LO$ linear orders: 
\[
xE_{\omega_{1}}y\Leftrightarrow(x,y\notin WO)\vee(ot(x)=ot(y)).
\]
We fix a tree $T$ inducing $E_{\omega_{1}}$ : $(x,y,f)$ is in $T$
if and only if either $f$ is an isomorphism between $x$ and $y$
or $f$ codes two $\omega$-decreasing sequences -- one in $x$ and
the other in $y.$ It is not hard to show that $x\in WO_{\alpha}\Rightarrow\delta(x)=\alpha$,
and $x\notin WO\Rightarrow\delta(x)=\infty.$ 
\item Let $A\subseteq\omega^{\omega}$ be a strictly analytic set. Given
$(x_{1},x_{2}),(y_{1},y_{2})\in\omega^{\omega}\times\{0,1\}$, $(x_{1},x_{2})E(y_{1},y_{2})$
if and only if 
\[
\left(x_{1}=y_{1}\in A\right)\vee\left((x_{1},x_{2})=(y_{1},y_{2})\right)
\]
 The equivalence relation $E$ is strictly analytic and all its classes are finite. Fix
a tree $T$ such that
\[
x\in^{\sim}A\Leftrightarrow T_{x}\in WF
\]
and for each $\alpha$ countable, let
\[
x\in B_{\alpha}\Leftrightarrow T_{x}\in WF_{\alpha}.
\]
The sets $B_{\alpha}$ are Borel and $^{\sim}A=\bigcup_{\alpha<\omega_{1}}B_{\alpha}.$

We can now use the tree $T$ to define another tree inducing the equivalence
relation $E$, and consider the rank associated with the new tree.
It is then easy to see that $x\in B_{\alpha}\Rightarrow\delta(x)=\alpha$,
and $x\in A\Rightarrow\delta(x)=\omega$. 

That example shows that the rank might be arbitrarily higher than the complexity
of the equivalence class.
\end{enumerate}
\end{example}

\subsection{Rectangular Borel canonization of Proper ideals}

Having those definitions in mind, one can try and prove rectangular
Borel canonization of proper ideals in the following way: 

\begin{itemize}
\item Fix a countable elementary submodel $M\preceq H_{\theta}$ for $\theta$
large enough, and force with $\mathbb{P}_{I}$ over $M$.
\item Show that $A_{x_{G}}$ is Borel in $M[x_{G}]$ and so
\[
M[x_{G}]\models\delta(x_{G})\leq\alpha
\]
for some $\alpha<\omega_{1}^{M[x]}=\omega_{1}^{M}$ (recall that $\mathbb{P}_{I}$
preserves $\omega_{1}$).
\item Use absoluteness to show that $\mathbb{V}\models\delta(x_{G})\leq\alpha$.
\item Use properness to guarantee that the set of $M$-generics is $I$-positive, and the above arguments to conclude that all of them
has rank less than $\omega_{1}^{M}<\omega_{1}.$
\end{itemize}

However, the 2nd and 3rd steps are in general impossible. Although
$A$ has only Borel sections, that statement is $\Pi_{4}^{1}$ (see
proposition \ref{definability_of_rank}), hence one must work harder
to show its preservation. The 3rd step provides us with another absoluteness
challenge, since $\Pi_{2}^{1}$ absoluteness between a submodel $N$ and the universe is guaranteed when $N$ contains all countable
ordinals, whereas $M[x_{G}]$ is countable.

The following proof follows the above lines and takes advantage of
the measurable cardinal to overcome the above mentioned difficulties.
We remind that by a theorem of Martin and Solovay (15.6 in \cite{kanamori}),
when there is a measurable cardinal $\kappa$, forcing notions of
cardinality less than $\kappa$ preserve $\Sigma_{3}^{1}$
statements.

\begin{thm}
Assume a measurable cardinal exists. Then proper ideals have rectangular Borel
canonization and ccc ideals have strong rectangular Borel canonization.
\end{thm}

\begin{proof}
The idea is as follows: given $U$ a $\kappa$-complete ultrafilter
on $\kappa$, one can form iterated ultrapowers of the universe, $\mathbb{V_{\alpha}},$
all well founded by a theorem of Gaifman. The same operation can be
applied on $M$ a countable elementary submodel of the universe such
that $U\in M$. Since the sequence $j^{(\alpha)}(\kappa)$ is increasing
and continuous, $M_{\omega_{1}},$ the $\omega_{1}'th$ iterated ultrapower
of $M$, contains all countable ordinals, so that $M_{\omega_{1}}$
and the universe agree on $\mathbf{\Pi_{2}^{1}}$ statements. On the
other hand, $M_{\omega_{1}}$ is an iterated ultrapower of $M$, so
they agree on all statements -- there is an elementary embedding between
them. We will then have enough absoluteness to conclude the proof.

So let $M\preceq H_{\theta}$ for $\theta$ large enough be a countable elementary submodel such that $\kappa\in M$
is measurable and $M$ contains all the relevant information. Fix $U\in M$ a $\kappa$-complete
ultrafilter on $\kappa$, and force with
$\mathbb{P}_{I}$ over $M$. Levy-Solovay theorem guarantees that $U$
remains a $\kappa$-complete ultrafilter in $M[x_{G}]$. For convenience,
denote $M[x_{G}]$ by $N$, remembering that $\omega_{1}^{N}=\omega_{1}^{M}$
because $\mathbb{P}_{I}$ is proper. We can then use $U$ to iterate
ultrapowers of both $\mathbb{V}$ and $N$ over all ordinals. Denote
by $\mathbb{V}_{\alpha}$ and $N_{\alpha}$ the $\alpha'th$ iterated
ultrapowers of $\mathbb{V}$ and $N$, respectively. The $V_{\alpha}'s$
are well-founded, and since $N_{\alpha}\subseteq\mathbb{V}_{\alpha}$,
the $N_{\alpha}'s$ are well founded as well, so we identify them
with their transitive collapses. Since $(j_{\alpha})^{N}(\kappa)$
is a normal sequence, $N_{\omega_{1}}$ has all countable ordinals.
Hence, as stated above, $N_{\omega_{1}}$ and $N$ are elementarily
equivalent, and $N_{\omega_{1}}$ and $\mathbb{V}$ are $\mathbf{\Pi_{2}^{1}}$
equivalent.

By the assumption, $\mathbb{V}\models A_{x_{G}}\ Borel$, and so there
is a countable ordinal $\alpha$ such that $\mathbb{V}\models\delta(x_{G})\leq\alpha$.
We would like this statement to be true in $N_{\omega_{1}}$, but
it is meaningless there: Although $\alpha$ is an element of $N_{\omega_{1}}$, it is not necessarily countable in $N_{\omega_{1}}$. The natural
solution will be collapsing $\alpha$ over $N_{\omega_{1}}$. The
resulting model, $N_{\omega_{1}}[Coll(\omega,\alpha)]$, still contains
all ordinals countable in $\mathbb{V},$ and also knows that $\alpha$
is countable, so we can finally reflect the statement $\delta(x_{G})\leq\alpha$
to get that $N_{\omega_{1}}[coll(\omega,\alpha)]\models\delta(x_{G})\leq\alpha$
and 
\[
N_{\omega_{1}}[coll(\omega,\alpha)]\models A_{x_{G}}\ Borel.
\]
 Note that in $N_{\omega_{1}},$ $\alpha$ is under a measurable cardinal,
hence by Martin-Solovay's theorem, collapsing $\alpha$ over $N_{\omega_{1}}$
preserves $\Sigma_{3}^{1}$ statements. Proposition \ref{definability_of_rank}
then assures that $N_{\omega_{1}}\models A_{x_{G}}\ Borel.$ Since
$N_{\omega_{1}}$ is elementarily equivalent to $N$, we have so far
shown that
\[
N\models A_{x_{G}}\ Borel,
\]
which means that $N\models\delta(x_{G})\leq\alpha$ for some $\alpha<\omega_{1}^{N}=\omega_{1}^{M}$.
Another use of the elementary equivalence of $N$ and $N_{\omega_{1}}$
proves that $N_{\omega_{1}}\models\delta(x_{G})\leq\alpha$, from
which $\Sigma_{2}^{1}$ absoluteness guarantees
\[
\mathbb{V}\models\delta(x_{G})\leq\alpha<\omega_{1}^{M}.
\]
Taking $B$ to be the set of $M$-generics concludes the proof.
Notice that if $I$ is ccc, $B$ is co-$I$.
\end{proof}

\subsection{Rectangular Borel canonization of provably ccc ideals}

We follow Stern's definitions and results from \cite{stern} . By
an $\alpha$-Borel code, for $\alpha$ a not necessarily countable
ordinal, we mean a well founded tree on $\alpha$ whose maximal points are associated
with basic open sets. An $\alpha$-Borel code naturally codes a
set generated from basic open sets by unions and intersections of
length at most $\alpha$. If $\alpha$ is countable, the set coded
by an $\alpha$-Borel code is Borel.

For a countable ordinal $\gamma<\omega_{1}$ , $L[\gamma]$ stands
for $L[a]$ where $a$ codes a well order of $\omega$ of order type
$\gamma$.

\begin{thm}
(Stern \cite{stern}) If $A$ is $\mathbf{\Pi}_{\gamma}^{0}\cap\Pi_{1}^{1}(z)$
, then $L[z,\gamma]$ has an $\omega_{\gamma}^{L[z,\gamma]}$-Borel
code for $A$.
\end{thm}

\begin{prop}
Let $A$ be a $\Sigma_{1}^{1}(z)$ subset of the plane with $\mathbf{\Pi_{\gamma}^{0}}$
sections. Let $I$ be a $\sigma$-ideal proper in $L[z,\gamma]$, and $x$ generic over $L[z,\gamma]$.
Then
\[
\delta(x)<\omega_{\gamma+1}^{L[z,\gamma]}.
\]
\end{prop}

\begin{proof}
Since $\mathbb{V}\models A_{x}\ is\ \mathbf{\Pi_{\gamma}^{0}}$, using
Stern's theorem we know that $L[z,x,\gamma]\models A_{x}\ is\ \omega_{\gamma}^{L[z,x,\gamma]}-Borel.$
Collapsing $\omega_{\gamma}^{L[z,x,\gamma]}$ over $L[z,x,\gamma]$
, we have: 
\[
L[z,x,\gamma][Coll(\omega,\omega_{\gamma}^{L[z,x,\gamma]}]\models A_{x}\ Borel.
\]

$\omega_{1}$ of the new model is $\omega_{\gamma+1}^{L[z,x,\gamma]}$.
Since $x$ is assumed to be $L[z,\gamma]$-generic and $\mathbb{P}_{I}$
doesn't collapse cardinals in $L[z,\gamma]$: 
\[
\omega_{\gamma+1}^{L[z,\gamma]}=\omega_{\gamma+1}^{L[z,\gamma][x]}.
\]
Hence there must be an $\alpha<\omega_{\gamma+1}^{L[z,\gamma]}$ such
that
\[
L[z,x,\gamma][Coll(\omega,\omega_{\gamma}^{L[z,x,\gamma]}]\models\delta(x)\leq\alpha.
\]
Shoenfield's absoluteness concludes the proof.
\end{proof}

\begin{thm}
Assume $\omega_{1}$ is inaccessible to the reals, and $I$ is ccc
in $L[z]$ for any real $z$. Then $I$ has strong rectangular Borel canonization
of analytic sets all of whose sections are $\mathbf{\Pi_{\gamma}^{0}}$ for some
$\gamma<\omega_{1}$.
\end{thm}

Note that part of the assumption here is that $I$ is defined
and a $\sigma$-ideal in $L[z]$ for any real $z$.

\begin{proof}
Let $A$ be a $\Sigma_{1}^{1}(z)$ set with $\mathbf{\Pi_{\gamma}^{0}}$
sections. Since $I$ is ccc in $L[z,\gamma]$, the set of generics over $L[z,\gamma]$
is co-$I$. $\omega_{1}$ is inaccessible in $L[z,\gamma]$, so
that in particular $\omega_{\gamma+1}^{L[z,\gamma]}<\omega_{1}.$
The previous proposition then concludes the proof.
\end{proof}

\section{\label{sec:Examples-and-Counterexamples}Examples and Counterexamples}

The following section elaborates on Borel canonization of equivalence
relations in its most general form:

\selectlanguage{american}%
\begin{problem}
Given an analytic equivalence relation $E$ on a Polish space $X$
and a $\sigma$-ideal $I$, does there exist an $I$-positive
Borel set $B$ such that $E$ restricted to $B$ is Borel?
\end{problem}

As mentioned before, in general the answer is negative. We will list
a few examples and counterexamples we find interesting.

\subsection{Non Borel classes and improper ideals}

\begin{example}
\label{cex-non Borel classes}(Kanovei, Sabok, Zapletal \cite{ksz})
Fix $K$ a strictly analytic Borel ideal on $\omega$ and define an
equivalence relation on $(2^{\omega})^{\omega}$ by:
\[
\bar{x}E\bar{y}\Leftrightarrow\{n\ :\ \bar{x}(n)\neq\bar{y}(n)\}\in K.
\]
Not only that $E$ is non Borel -- one can easily show that none of
its classes is Borel. Now consider the following $\sigma$-ideal
$I$ on $(2^{\omega})^{\omega}$ : $A\notin I$ if $A$ does not contain
a set of the form $\Pi_{n\in\omega}P_{n}$ for $P_{n}$ perfect sets.
In \cite{ksz} it is shown that $I$ is a $\sigma$-ideal, and
even a proper one. However, for any $B$ Borel and $I$-positive,
$E$ can be reduced to $E\restriction_{B}$ (since $B$ contains a
copy of the whole space). In particular, $E\restriction_{B}$ is not
Borel.
\end{example}

The above example clarifies why we assume all classes are Borel, and
why that assumption is not redundant even when one assumes the properness
of the ideal $I$.

\begin{example}
\label{cex-orbit-er}Let $E$ be an analytic and non Borel \textbf{orbit
equivalence relation} on $\omega^{\omega}$ . Consider the following
$\sigma$-ideal : $C\in I$ if there is $B\supseteq C$ Borel such
that $E\restriction_{B}$ is Borel. 
\begin{enumerate}
\item $\omega^{\omega}\notin I$, since $E$ is non Borel. Trivially enough,
$\emptyset\in I$.
\item $I$ is downward closed.
\item To show that $I$ is $\sigma$-closed , consider $\langle C_{n}\ :\ n\in\omega\rangle$
a sequence of sets in $I$, and let $\langle B_{n}\ :\ n\in\omega\rangle$
Borel such that $B_{n}\supseteq C_{n}$ and $E\restriction_{B_{n}}$
is Borel. Since $\bigcup B_{n}\supseteq\bigcup C_{n}$, we will be
satisfied showing that $E\restriction_{\bigcup B_{n}}$ is Borel.
That follows easily from Hjorth analysis (see \cite{hjorth-paper,my-hjorth}):
Let $\delta_{n}$ be some countable ordinal bounding the Hjorth rank
on $B_{n}$. Then $sup_{n}(\delta_{n})$ bounds the rank on $\bigcup_{n}B_{n}$,
hence $E\restriction_{\bigcup B_{n}}$ is Borel.
\item There is no Borel canonization of $E$ with respect to $I$: if $B$
is Borel and $I$ positive then by the very definition of $I$, $E\restriction_{B}$
is non Borel.
\end{enumerate}
\end{example}

\begin{rem}
In 3 we have used the fact that $E$ is an orbit equivalence relation.
We conjecture it is not necessarily true for general analytic and
non Borel equivalence relations.
\end{rem}

In example \ref{cex-orbit-er}, Borel canonization fails although
all classes are Borel. However, the $\sigma$-ideal considered here
is non proper. The sets $\mathcal{A}_{\alpha}=\{x\ :\ \delta(x)\leq\alpha\}$
,where $\delta$ is the Hjorth rank, are Borel and $I$-small.
Hence a generic element will have rank greater or equal than $\omega_{1}$,
clearly collapsing $\omega_{1}.$ That example thus indicates the
necessity of assuming the properness of the $\sigma$-ideal $I$.

\subsection{Perfect set properties of equivalence relations}

\begin{defn}
Let $E$ be an equivalence relation on a Polish space $X$. $E$ has
perfectly many classes if there is a perfect set $P\subseteq X$ of
pairwise inequivalent elements.
\end{defn}

One of the most well known results in the study of equivalence relations
in set theory is the following theorem due to Silver:

\begin{thm}
(Silver) Let $E$ be a coanalytic equivalence relation on a Polish
space $X$. Then either $E$ has countably many classes, or it has
perfectly many classes.
\end{thm}
\selectlanguage{english}%

Silver's theorem fails for analytic equivalence relations:

\begin{enumerate}
\item For $x,y\in LO$ linear orders, let 
\[
xE_{\omega_{1}}y\Leftrightarrow x,y\notin WO\ \vee\ ot(x)=ot(y).
\]
Then $E_{\omega_{1}}$ has uncountably many classes -- $\langle WO_{\alpha}\ :\ \alpha<\omega_{1}\rangle$
and the class of ill orders. However, $E_{\omega_{1}}$ does not have
perfectly many classes, as by the boundedness theorem perfect sets
in $WO$ will have bounded order type. Note that all but one of the
equivalence classes are Borel.
\item For $x,y\in\omega^{\omega}$, let 
\[
xE_{ck}y\Leftrightarrow\omega_{1}^{ck(x)}=\omega_{1}^{ck(y)}.
\]
 Then $E_{ck}$ is analytic with uncountably many classes. The effective
version of the boundedness theorem demonstrates that $E_{ck}$ does
not have perfectly many classes. Notice that all the $E_{ck}$ classes
are Borel.
\item Given a Polish group action $(G,X)$ inducing a non Borel orbit equivalence
relation, let 
\[
xE_{\delta}y\Leftrightarrow\delta(x)=\delta(y),
\]
where $\delta$ is the Hjorth rank (as in \cite{hjorth-paper,my-hjorth}).
$E_{\delta}$ is analytic with uncountably many classes. It does not
have perfectly many classes -- otherwise we could have ccc forced $\neg CH$,
and use Shoenfield's absoluteness to get in the generic extension
a perfect set of size less than the continuum. Notice that here as
well, all the $E_{\delta}$ classes are Borel.
\end{enumerate}

\begin{prop}
Let $E$ be analytic with uncountably many classes but not perfectly
many. Let $I_{E}$ be the $\sigma$-ideal generated by the equivalence
classes. Then for any $B$ Borel $I_{E}$-positive, $E\restriction_{B}$
is non Borel.
\end{prop}

\begin{proof}
Let $B$ be Borel $I_{E}$-positive. Since the equivalence classes are $I_E$-small, $B$ must intersect uncountably
many classes. If $E\restriction_{B}$ was Borel, Silver's theorem
would produce a perfect set of pairwise inequivalent elements -- contradicting
the assumptions on $E$.
\end{proof}

Hence $E_{\omega_{1}},E_{ck}$ and $E_{\delta}$ together with their
induced $\sigma$-ideals all serve as counterexamples -- the first
to Borel canonization of analytic equivalence relations, and the 2nd
and 3rd to Borel canonization of analytic equivalence relations with
Borel classes. Fortunately, $I_{E_{\omega_{1}}},I_{E_{ck}}$ and $I_{E_{\delta}}$
are all improper -- in fact $\mathbb{P}_{I_{E_{\omega_{1}}}},\mathbb{P}_{I_{E_{ck}}}$
and $\mathbb{P}_{I_{E_{\delta}}}$ all collapse $\omega_{1}:$

\begin{enumerate}
\item Given $\gamma<\omega_{1}$, the set $W_{\gamma}=\{x\ :\ \forall k\ ot(x\restriction_{k})<\gamma\}$
is in $I_{E_{\omega_{1}}}.$ The generic real must avoid all of them,
hence its well founded part has order type greater or equal than $\omega_{1}$
 -- thus collapsing $\omega_{1}$.
\item Let $x_{G}$ be the generic real added by forcing with $\mathbb{P}_{I_{E_{ck}}}$.
Then $\omega_{1}^{ck(x_{G})}\geq\omega_{1}$.
\item Let $x_{G}$ be the generic real added by forcing with $\mathbb{P}_{I_{E_{\delta}}}$.
Then $\delta(x_{G})\geq\omega_{1}$.
\end{enumerate}

This is no coincidence: In \cite{my-psp} we show that for $E$ analytic
with uncountably many classes but not perfectly many, $I_{E}$ is
improper.

When considering the above equivalence relations with proper ideals,
Borel canonization is trivially found -- we show it for $E_{ck},$ proofs
for the other two are almost the same.

\begin{example}
Consider $E_{ck}$ and a proper ideal $I$. Since $\mathbb{P}_{I}$ doesn't collapse $\omega_{1}$, there must be
some $\alpha<\omega_{1}$ such that $\{x\ :\ \omega_{1}^{ck(x)}=\alpha\}$
is $I$-positive. $E_{ck}$ restricted to that Borel set is trivial.
\end{example}

\subsection{Borel canonization of $\mathbf{\Delta}_{\mathbf{2}}^{\mathbf{1}}$
sets}

We end this section considering equivalence relations which are less
definable. By Borel canonization we still mean -- ``$E\restriction_{B}$
is Borel for some $B$ Borel $I$-positive set'' or $"A\cap(B\times B)$
is Borel for some $B$ Borel $I$-positive set'', etc.

\begin{thm}
\cite{chan}
\label{negative result delta_1_2}In $L$, there is a countable $\Delta_{2}^{1}$
equivalence relation that does not have perfectly many classes. In particular, in $L$ $\sigma$-ideals do not have Borel canonization of $\Delta^1_2$ equivalence relations with Borel classes.
\end{thm}

\begin{proof}
In $L$, consider the following equivalence relation: 
\[
xEy\Leftrightarrow\left(\forall\alpha\ admissible \ x\in L_{\alpha}\Leftrightarrow y\in L_{\alpha}\right).
\]
Since the constructibility rank of $x$ and the admissibility of ordinals are decided by a countable
model and by all countable models, $E$ is a $\Delta_{2}^{1}$ equivalence relation. All $E$
classes are countable, since all $L_{\alpha}'s$ are. We will show
that any perfect tree $T$ must have two equivalent elements.

Let $T\in L$ be perfect, and let $\alpha$ be such that $T \in L_{\alpha}$. Let $\beta$ be the first admissible ordinal greater then $\alpha$ such that $L_\beta$ has a real not in $L_\alpha$. Using \cite{chan} fact 9.5, $L_\alpha$ is countable in $L_\beta$. Since $T$ has uncountably many branches in $L_\beta$, there must be 
\[x \neq y\in [T]\cap L_\beta\]
that are not in $L_\alpha$. It follows that $x$ and $y$ are equivalent.
\end{proof}

The following proposition is weaker, but its proof is easier:

\begin{prop}
In $L$, $\sigma$-ideals do not have square Borel canonization
of $\mathbf{\Delta_{2}^{1}}$ sets with Borel sections.
\end{prop}

\begin{proof}
Denote by $<_{L}$ the $\Delta_{2}^{1}$ well order of the reals in
$L$, whose horizontal and vertical sections are all Borel. The set $<_{L}$
does not have square Borel canonization with respect to any $\sigma$-ideal. This is because $<_{L}$ restricted to an uncountable set
is an order of length $\omega_{1},$ hence by the boundedness theorem
for analytic well founded relations, it cannot be analytic.
\end{proof}

All the above can be done in $L[z]$ for $z$ real,
and in any model in which $\mathbb{R}^{L[z]}=\mathbb{R}$ for some
$z\in\mathbb{R}.$

\section{Counterexamples to Rectangular Borel Canonization}

Counterexamples are implicit in \cite{ikegami-2nd}:

\begin{example}
\cite{ikegami-2nd}
Consider the meager ideal and theorem \ref{thm:(-Ikegami-)}. For
that ideal, rectangular Borel canonization and strong rectangular
Borel canonization are equivalent (see remark \ref{remark on ccc and borel canonization}).
Hence theorem \ref{thm:(-Ikegami-)} provides counterexamples when
not all $\mathbf{\Sigma_{2}^{1}}$ sets have the Baire property. The
same is true for the null ideal.
\end{example}

\begin{prop}
In $L$, proper ideals do not have rectangular Borel canonization
of analytic sets with Borel sections. The same is true for $L[z]$
where $z$ is a real. 
\end{prop}

\begin{proof}
The argument is based on example 2.3.5 of \cite{forcing-idealized}. Working in
$L$, let 
\[
(x,y)\in A\Leftrightarrow x\in L_{\omega_{1}^{ck(y)}}.
\]
The set $A$ is coanalytic with Borel vertical sections, since given $x\in L_{\alpha}$
and $\alpha$ minimal with that property, 
\[
A_{x}=\{y\ :\ x\in L_{\omega_{1}^{ck(y)}}\}=\{y\ :\ \omega_{1}^{ck(y)}\geq\alpha\},
\]
which is Borel. By way of contradiction, fix $B$ Borel $I$-positive
such that $A\cap(B\times\omega^{\omega})$ is Borel. Using $\mathbb{P}_{I}$
-uniformization (2.3.4 of \cite{forcing-idealized}) there exists $C\subseteq B$
Borel $I$-positive and $f:C\to\omega^{\omega}$ Borel such that
$f\in L$ and $f\subseteq A$. Let $x_{G}\in C$ be a Sacks real over
$L$. By analytic absoluteness, $L[x_{G}]\models f\subseteq A$, and
in particular, $(x_{G},f(x_{G}))\in A$, contradicting the fact that
$x$ is not constructible.
\end{proof}

Let $A$ be an analytic subset of the plane and $B\subseteq\omega^{\omega}$
Borel $I$-positive subset of reals such that $A\cap(B\times\omega^{\omega})$
is Borel. Then using Shoenfield's absoluteness, $\mathbb{P}_{I}\Vdash A\cap(B\times\omega^{\omega})\ Borel$, and in particular
\[
B\Vdash A_{x_{G}}\ Borel.
\]

Hence an ideal $I$ such that $\mathbb{P}_{I}$ adds a non Borel section
is a counterexample to rectangular Borel canonization. We now show
that even under mild large cardinal assumptions, there might exist such an ideal
which is ccc:

\begin{fact}
\label{fact_force_omega_1}If $\omega_{1}$ is inaccessible to the
reals and is not Mahlo in $L$, then there is a ccc forcing adding
a real $x$ such that $\omega_{1}^{L[x]}=\omega_{1}$.
\end{fact}

For the proof, see theorem 6 of \cite{bagaria-friedman-absoluteness}.

\begin{prop}
If $\mathbb{P}$ is a ccc forcing adding a real $x$, then there is
a ccc ideal $I$ such that $\mathbb{V}[x]\subseteq\mathbb{V}^{\mathbb{P}}$
is a $\mathbb{P}_{I}$ extension and $x$ is the $\mathbb{P}_I$ generic real.
\end{prop}

\begin{proof}
Fix $\tau$ a $\mathbb{P}$-name for the real $x$. For $B$ Borel,
define
\[
B\in I\Leftrightarrow\mathbb{P}\Vdash\tau\notin B.
\]
$I$ is a $\sigma$-ideal (in fact, a $\sigma$-ideal on Borel
sets which generates a $\sigma$-ideal). We claim that it is ccc.
Let $\langle B_{\alpha}\ :\ \alpha<\omega_{1}\rangle$ be an antichain
of $I$-positive sets, which is, for $\alpha_{1}\neq\alpha_{2}$,
\[
\mathbb{P}\Vdash\tau\notin\left(B_{\alpha_{1}}\cap B_{\alpha_{2}}\right).
\]
 Fix $p_{\alpha}\in\mathbb{P}$ such that $p_{\alpha}\Vdash\tau\in B_{\alpha}.$
Then $\langle p_{\alpha}\ :\ \alpha<\omega_{1}\rangle$ must be an
antichain, hence countable, as we have hoped.

In $\mathbb{V}^{\mathbb{P}},$ the generic $x$, as a realization
of $\tau$, avoids all Borel $I$-small sets of the ground model,
hence it is $\mathbb{P}_{I}$ generic over $\mathbb{V}$. Thus $\mathbb{V}[x]$
is the promised $\mathbb{P}_{I}$ extension.
\end{proof}

\begin{thm}
\label{cex_rectangular_borel_canonization_adding_nonborel_section}If
$\omega_{1}$ is inaccessible to the reals and is not Mahlo in $L$,
then there is a ccc ideal $I$ not having rectangular Borel canonization
of analytic sets with Borel sections. Moreover, $\mathbb{P}_{I}\Vdash A_{x_{G}}\ non\ Borel$
for some $A$ analytic with Borel sections.
\end{thm}

\begin{proof}
For every $x$, in $L[x]$ there exists a $\Pi_{1}^{1}(x)$ uncountable
set with no perfect subset. Fix $\Phi(x)$ a $\Pi_{1}^{1}(x)$ formula
defining that set. Then in any universe $\mathbb{V},$ $\Phi(x)$
defines a subset of $L[x]$ of size $\omega_{1}^{L[x]}$ with no perfect
subset. Moreover, the definition is uniform -- there is a $\Pi_{1}^{1}$
formula $\Psi(x,y)$ such that for every $x$, 
\[
\{y\ :\ \Psi(x,y)\}
\]
is a subset of $L[x]$ of size $\omega_{1}^{L[x]}$ with no perfect
subset.

Consider the subset of the plane defined by $\Psi$. The vertical
sections of $\Psi$ are either countable or strictly coanalytic -- 
since we assume $\omega_{1}$ is inaccessible to the reals, they are
all countable and in particular Borel. Use the forcing of proposition
\ref{fact_force_omega_1} to obtain a ccc extension $\mathbb{V}^{\mathbb{P}}$
with $x\in\mathbb{V}^{\mathbb{P}}$ such that $\omega_{1}^{L[x]}=\omega_{1}.$
Use the previous proposition to construct a ccc ideal $I$ such that
\[
\mathbb{V}[x]=\mathbb{V}^{\mathbb{P}_{I}}.
\]
 Obviously, $\mathbb{P}_{I}\Vdash\omega_{1}^{L[x_{G}]}=\omega_{1}$
for $x_{G}$ its generic real. In particular, $\Psi$ has a new section
which is non Borel, and rectangular Borel canonization fails.
\end{proof}

\begin{rem}
The reader is encouraged to compare the above example with the positive
results of previous sections. When doing so, note that $I$ is not
even defined in $L$ -- its definition requires a club in $\omega_{1}$
of ordinals which are singular cardinals in $L$.
\end{rem}

\begin{cor}
Rectangular Borel canonization for ccc ideals implies that $\omega_1$ is inaccessible to the reals and Mahlo in L.
\end{cor}

\begin{proof}
 By theorem \ref{thm_Ikegami_more_precise}, Hechler ideal has rectangular Borel canonization if and only if $\mathbf{\Sigma^1_2}$ sets are Hechler measurable. In \cite{brendle-lowe} it is shown that measurability of $\mathbf{\Sigma^1_2}$ sets with respect to the Hechler ideal is equivalent to $\omega_1$ being inaccessible to the reals. To see that $\omega_1$ is Mahlo in $L$, use the previous theorem.
\end{proof}

The case of square Borel canonization is different -- for $A$ analytic
and $B$ Borel, if $A\cap(B\times B)$ is Borel, then $\mathbb{P}_{I}\Vdash A\cap(B\times B)\ is\ Borel$
, hence
\[
B\Vdash(A_{x_{G}}\cap B)\ is\ Borel.
\]
In order to construct a counterexample, we can try and find $A$ and $I$
such that no $B$ Borel $I$-positive forces the Borelness of $A_{x}\cap B$:

\begin{problem}
Let $\Psi$ and $I$ be as in theorem \ref{cex_rectangular_borel_canonization_adding_nonborel_section},
and let $A$ be the coanalytic subset of the plane defined by $\Psi.$
Can we find $B\in\mathbb{P}_{I}$ such that $B\Vdash(A_{x_{G}}\cap B)\ is\ Borel$?
\end{problem}

\subsection{Non absoluteness of ``All classes are Borel''}

The previous example shows that for $A$ an analytic subset of the
plane, the property ``all vertical sections of $A$ are Borel''
can be forced false by a ccc ideal. The same applies for analytic
equivalence relations:

\begin{prop}
\label{non_absoluteness_of_Borel_sections}There is an analytic equivalence
relation $E$ such that:
\begin{enumerate}
\item If $\omega_{1}$ is inaccessible to the reals and is not Mahlo in
$L$, then all $E$ classes are Borel and there is a ccc ideal $I$
such that
\[
\mathbb{P}_{I}\Vdash[x_{G}]\ is\ non\ Borel.
\]
\item If $\omega_{1}$ is inaccessible to the reals, then all $E$ classes
are Borel, while in $L$ there is a non Borel class.
\end{enumerate}
\end{prop}

\begin{proof}
We use a variation of the example introduced by theorem \ref{cex_rectangular_borel_canonization_adding_nonborel_section}.

Let $\Psi(x,y)$ be as in theorem \ref{cex_rectangular_borel_canonization_adding_nonborel_section} -- a $\Pi_{1}^{1}$ formula whose vertical sections
are subsets of $L[x]$ of size $\omega_{1}^{L[x]}$ with no perfect
subset. Let 
\[
(x_{1},y_{1})E(x_{2},y_{2})\Leftrightarrow(x_{1}=x_{2})\wedge\left(\left((\neg\Psi(x_{1},y_{1})\wedge\neg\Psi(x_{2},y_{2})\right)\vee(y_{1}=y_{2})\right)).
\]
$E$ is an analytic equivalence relation, and the equivalence class
of $(x_{0},y_{0})$ is either a singleton or 
\[
\{(x_{0},y)\ :\ \neg\Psi(x_{0},y)\}.
\]
Hence if $\neg\Psi(x_{0},y_{0}),$ $[(x_{0},y_{0})]_{E}$ is Borel
if and only if $\omega_{1}^{L[x_{0}]}<\omega_{1}$.

The 1st clause then follows using the forcing notion introduced in
the previous subsection, while the 2nd clause is obvious.
\end{proof}

\begin{rem}
Failure of downward absoluteness of ``all classes are Borel'' follows
from $ZFC$ alone: In $L$, fix $A$ a coanalytic uncountable set
without a perfect subset, and let 
\[
xEy\Leftrightarrow(x=y)\vee(x,y\notin A).
\]
The analytic equivalence relation $E$ has a non Borel class, but after collapsing
$\omega_{1}$ over $L$, all its classes become Borel.
\end{rem}

\begin{problem}
The nature of the above examples raises the following questions:
\begin{enumerate}
\item Is there an analytic equivalence relation with Borel classes in $L$
but non Borel classes under large cardinal assumptions?
\item Can we prove the failure of upward absoluteness of ``all classes
are Borel'' without using the consistency of an inaccessible cardinal?
\end{enumerate}
\end{problem}

We end this section by computing the complexity of various properties
discussed in this paper, the most important of them are "section $A_x$ is Borel" and "the rank of $x$ is less then $ot(f)$":

\begin{prop}
In what follows, we say that a property is $\Pi_{\alpha}^{1}$ / $\Sigma_{\alpha}^{1}$
if it is provably $\Pi_{\alpha}^{1}$ / $\Sigma_{\alpha}^{1}$, which
is, there is a lightface $\Pi_{\alpha}^{1}$ / $\Sigma_{\alpha}^{1}$
formula equivalent to the property in \uline{every} model of $ZFC.$
Then, assuming $con(ZFC+Inaccessible)$ and given $A$ a subset of
the plane:
\begin{enumerate}
\item If $A$ is $\Sigma_{1}^{1}$, ``all sections of $A$ are countable''
is $\Pi_{1}^{1}.$ 
\item If $A$ is $\Pi_{1}^{1}$, ``all sections of $A$ are countable''
is $\Pi_{4}^{1}$, and is neither $\Pi_{3}^{1}$ nor $\Sigma_{3}^{1}$.
\item If $A$ is $\Pi_{1}^{1}$, ``all sections of $A$ do not contain
a perfect set'' is $\Pi_{2}^{1}$ , hence is absolute between generic
extensions.
\item If $A$ is $\Sigma_{1}^{1}$ or $\Pi_{1}^{1}$, ``all sections are
Borel'' is $\Pi_{4}^{1}$ , and is neither $\Pi_{3}^{1}$ nor $\Sigma_{3}^{1}$.
The same is true when $A$ is an equivalence relation.
\item If $A$ is $\Sigma_{1}^{1}$ or $\Pi_{1}^{1}$ and $\alpha<\omega_{1}^{ck}$
, ``all sections are $\Pi_{\alpha}^{0}$'' is $\Pi_{4}^{1}$ , and
is neither $\Pi_{3}^{1}$ nor $\Sigma_{3}^{1}$.
\item If $A$ is $\Sigma_{1}^{1}$ or $\Pi_{1}^{1}$, the set $\{x\ :\ A_{x}\ is\ Borel\}$
is $\Sigma_{3}^{1}$, and not $\Pi_{3}^{1}$.
\item If $A$ is $\Sigma_{1}^{1}$ and $\delta$ is some rank associated
with it, then $\{(x,f)\ :\ f\in WO,\ \delta(x)\leq f\}$ is not $\Sigma_{2}^{1}$.
\end{enumerate}
\end{prop}

Most of the above can be relativized.

\begin{proof}
For $(1)$ , recall that a $\Sigma_{1}^{1}(x)$ set is countable if
and only if all its elements are hyperarithmetic in $x$. For $(2)$,
use theorem \ref{cex_rectangular_borel_canonization_adding_nonborel_section}, proposition \ref{non_absoluteness_of_Borel_sections} and
Shoenfield's absoluteness. Regarding $(6)$, note that if ``$A_{x}$
is Borel'' had been $\Pi_{3}^{1}$ then ``all sections are Borel'' would have been $\Pi_{3}^{1}$ as well. In $(7)$, a $\Sigma_{2}^{1}$ definition for
this set will produce a $\Sigma_{2}^{1}$ definition of ``$A_{x}$
is Borel''.
\end{proof}

\selectlanguage{american}%

\end{onehalfspace}
\end{document}